\documentclass[11pt,a4paper]{article}
\usepackage[T1]{fontenc}
\usepackage[utf8]{inputenc}
\usepackage[english]{babel}
\usepackage{lmodern,microtype}
\usepackage[margin=25mm,headheight=14pt]{geometry}
\usepackage{amsmath,amssymb,amsthm,mathtools,enumitem,booktabs,array}
\usepackage{xurl}
\usepackage[hidelinks,unicode]{hyperref}
\usepackage{fancyhdr}
\setlist[enumerate]{itemsep=2pt,topsep=4pt}
\newtheorem{theorem}{Theorem}[section]
\newtheorem{lemma}[theorem]{Lemma}
\newtheorem{proposition}[theorem]{Proposition}
\newtheorem{corollary}[theorem]{Corollary}
\theoremstyle{remark}

\newcommand{\F}{\mathbb{F}}
\newcommand{\om}{\omega}
\newcommand{\indab}{\mathrm{i}}
\newcommand{\rad}{\operatorname{rad}}
\newcommand{\rank}{\operatorname{rank}}
\newcommand{\Hom}{\operatorname{Hom}}
\newcommand{\eps}{\varepsilon}
\hypersetup{pdftitle={Eventual exactness and rigidity for abelian covers of groups},
 pdfauthor={Guillaume Lecomte},
 pdfsubject={Exact abelian covering numbers, extremal rigidity},
 pdfkeywords={abelian covers, noncommuting groups, symplectic geometry, extremal rigidity}}
\title{Eventual exactness and rigidity\\for abelian covers of groups}
\author{Guillaume Lecomte\thanks{ORCID: \href{https://orcid.org/0009-0003-3532-3379}{0009-0003-3532-3379}.}}
\date{11 October 2026}
\begin{document}
\maketitle

\begin{abstract}
For a group $G$, let $a(G)$ denote the minimum number of abelian
subgroups covering $G$, and let $\om(G)$ be the maximum size of a
pairwise noncommuting set. We prove the exact formula
\[
 \sup_{\om(G)\le n}a(G)=2^{\lfloor(n-1)/2\rfloor}+1
 \qquad(n\ge63).
\]
Equality holds precisely for groups with derived subgroup of order two
and binary symplectic central quotient of dimension
$2\lfloor(n-1)/2\rfloor$.
The proof combines centralizer descent with a structural separation
of nonscalar alternating commutator maps.
\end{abstract}


\section{Introduction}\label{sec:introduction}

Two measures of the noncommutativity of a group $G$ are the maximum
size $\om(G)$ of a pairwise noncommuting set and the minimum number
$a(G)$ of abelian subgroups needed to cover $G$. Necessarily
$\om(G)\le a(G)$, but the gap may be exponential. Erd\H{o}s asked
for its extremal size \cite{Erdos1997}. For $n\ge1$ put
\[
 h(n)=\sup_{\om(G)\le n}a(G),\qquad
 L_n=2^{\lfloor(n-1)/2\rfloor},
\]
where the supremum ranges over all groups and $a(G)=\om(G)=1$
for abelian groups. Equivalently, $a(G)$ and $\om(G)$ are the
chromatic and clique numbers of the noncommuting graph on
$G\setminus Z(G)$ when $G$ is nonabelian
\cite{FaberLaverMcKenzie1978,AbdollahiAkbariMaimani2006}.

Neumann \cite{Neumann} proved that finite $\om(G)$ forces
$[G:Z(G)]<\infty$. Pyber \cite{Pyber1987} established an
exponential index bound, while extraspecial groups supply an
exponential lower bound; hence $h(n)=2^{\Theta(n)}$.
Related covering questions appear in
\cite{Brown1988,Chin2005,PodoskiSzegedy2002,WangLiu2016}.
The author's earlier work \cite[Theorem~2.2]{Lecomte2026} gave
\[
 \log_2 h(n)=n/2+O\bigl(\sqrt n\,(\log(n+2))^3\bigr).
\]
We now determine $h(n)$ exactly for every $n\ge63$ and
characterize all extremal groups.

Here $G'$ is the derived subgroup, $Z(G)$ the centre, and
$C_d$ the cyclic group of order $d$.

\begin{theorem}[Exactness and rigidity]\label{thm:parity}
For every integer $n\ge63$, put $m=\lfloor(n-1)/2\rfloor$. Then
\begin{equation}\label{eq:parity-main}
 h(n)=L_n+1=2^m+1.
\end{equation}
A group $G$ with $\om(G)\le n$ attains this bound if and only if
\[
 G'\cong C_2,\qquad G/Z(G)\cong\F_2^{2m}.
\]
Such a group has $\om(G)=2m+1$.
\end{theorem}

The lower bound is realized by extraspecial $2$-groups. The
classification includes infinite groups and explains the parity
phenomenon $h(2m+1)=h(2m+2)=2^m+1$ in the stated range.

The proof has three steps. Section~\ref{sec:class2} establishes the
symplectic model and a prime-wise isotropic-index bound.
Section~\ref{sec:tools} uses a two-branch centralizer descent to
reduce the upper bound to class-two $2$-groups. In
Section~\ref{sec:parity}, scalarization and symplectic peeling
exclude all nonscalar commutator maps. The effective comparisons
are collected in Appendix~\ref{app:numerics}. We first prove the
upper bound for even $n\ge64$; the odd case follows from
$L_{n+1}=L_n$ and monotonicity of $h$.

\section{The symplectic model and the index bound}\label{sec:class2}

We begin with a finite model that preserves the covering number,
the clique number and the structural data in Theorem~\ref{thm:parity}.
We then express commutation as an alternating map and compare its
isotropic subgroups with those of the symplectic example.

\begin{lemma}[Finite reduction]\label{lem:finite-reduction}
For every group $G$ with $\om(G)<\infty$, there is a finite group $Q$
such that
\[
 a(Q)=a(G),\qquad \om(Q)=\om(G),\qquad
 Q/Z(Q)\cong G/Z(G),\qquad Q'\cong G'.
\]
\end{lemma}
\begin{proof}
By Neumann's theorem \cite[Theorem~6]{Neumann}, $G/Z(G)$ is
finite. Choose representatives of its cosets and let $F$ be the
subgroup they generate. Thus $G=FZ(G)$. Restriction of an abelian
cover of $G$ to $F$ gives $a(F)\le a(G)$, while multiplying every
member of an abelian cover of $F$ by $Z(G)$ gives the reverse
inequality. Likewise, since central factors do not change
commutators, $\om(F)=\om(G)$. We also have
$F'=G'$ and $Z(F)=F\cap Z(G)$: an element centralizing $F$
automatically centralizes $FZ(G)=G$. Hence
$F/Z(F)\cong G/Z(G)$.

The finitely generated group $F$ has centre of finite index, so
$Z(F)$ is finitely generated; Schur's theorem \cite{Schur} makes
$F'$ finite. Write $Z(F)\cong\mathbb Z^r\oplus T$, with $T$ finite,
and choose a positive integer $k$ divisible by the exponent of $T$.
Put $N=Z(F)^k$ and $Q=F/N$. Then $N\cong k\mathbb Z^r$ is
torsion-free and $N\cap F'=\{1\}$. Since $Z(F)/N$ and $F/Z(F)$
are finite, so is $Q$.

For $x,y\in F$, centrality of $N$ gives
$[xN,yN]=[x,y]N$. As $[x,y]\in F'$ and $N\cap F'=\{1\}$,
\[
 [xN,yN]=1\quad\Longleftrightarrow\quad[x,y]=1.
\]
Consequently, pairwise noncommuting sets project and lift without
changing their maximal cardinality. Images of abelian subgroups
are abelian; conversely, the full preimage in $F$ of any abelian
subgroup of $Q$ is abelian by the displayed equivalence. Thus
$a(Q)=a(F)$ and $\om(Q)=\om(F)$.

Finally $Q'=F'N/N\cong F'$. If $xN\in Z(Q)$, then
$[x,F]\subseteq N\cap F'=1$, hence $x\in Z(F)$. Therefore
$Z(Q)=Z(F)/N$ and
$Q/Z(Q)\cong F/Z(F)\cong G/Z(G)$, as required.
\end{proof}

\paragraph{Conventions.}
Except when explicitly stated otherwise, groups are finite from
this point onward. Write $\indab(G)$ for the minimum index of an
abelian subgroup, $[x,y]=x^{-1}y^{-1}xy$ and $x^y=y^{-1}xy$.
For $S\subseteq G$, its centralizer is $C_G(S)$ and $\om(S)$
denotes its maximum pairwise noncommuting subset size.
``Class two'' means nilpotency class at most two. Unsubscripted logarithms are natural.

\subsection{Commutator maps and the scalar model}

For class-two groups, commutation is encoded by an alternating map on
the central quotient \cite[Section~2]{BrooksbankMaglioneWilson2017}.

Let $\F_q$ denote the field of order $q$. Fix a prime $p$,
and consider a biadditive alternating map
$\beta:V\times V\to W$ between finite abelian $p$-groups,
written additively. Thus $\beta(x,x)=0$ for all $x\in V$.
We call $x,y$ \emph{orthogonal} if $\beta(x,y)=0$, and a subgroup
$A\le V$ \emph{isotropic} if $\beta(A,A)=0$.
The analogues of our group invariants are $a(\beta)$, the
minimum size of an isotropic cover of $V$;
$\om(\beta)$, the maximum size of a pairwise nonorthogonal set;
and $p^{\tau(\beta)}$, the minimum index of an isotropic subgroup.

For $R\le V$, let $\beta|_R$ denote restriction to
$R\times R$, and put
$C_R(X)=\{r\in R:\beta(r,X)=0\}$,
$C_R(x)=C_R(\{x\})$ and $\rad\beta=C_V(V)$.
Passing to $V/\rad\beta$ preserves all three invariants:
cliques project and lift, while isotropic subgroups can be enlarged
by the radical. We call $\beta$ \emph{reduced} when
$\rad\beta=0$ and its values generate $W$.
For subsets $X,Y\subseteq V$, the symbol $\beta(X,Y)$ denotes
the subgroup generated by their pairings.

For a class-two $p$-group $P$, set $V=P/Z(P)$ and $W=P'$.
The commutator map $\beta_P:V\times V\to W$ is alternating,
has zero radical, and its values generate $W$.
Abelian subgroups containing $Z(P)$ correspond exactly to isotropic
subgroups of $V$. Enlarging any abelian subgroup by $Z(P)$ shows that
$a(P)=a(\beta_P)$ and $\indab(P)=p^{\tau(\beta_P)}$;
commutator-preserving projection and lifting give
$\om(P)=\om(\beta_P)$.

For an abelian direct factor $A$, these invariants are preserved:
covers extend and restrict, commutation depends only on the
$P$-coordinates, and
$|B|\le|\operatorname{pr}_P B|\,|A|$ for abelian $B\le P\times A$.

The binary model has the following exact invariants; related clique
bounds appear in \cite{WangLiu2016}.
The covering calculation also appears in \cite[Lemma~4.6]{Lecomte2026}.

\begin{lemma}[The classical extraspecial example]\label{lem:extraspecial}
For each $m\ge1$, an extraspecial group $E_m$ of order $2^{1+2m}$ satisfies
\[
 \om(E_m)=2m+1,\qquad a(E_m)=2^m+1,\qquad \indab(E_m)=2^m.
\]
\end{lemma}
\begin{proof}
The commutator induces a nondegenerate alternating form $b$ on
$V=E_m/Z(E_m)\cong\F_2^{2m}$. Since commutation depends only
on the images in $V$, it suffices to compute the corresponding
invariants of $b$.

For concreteness, one may take the binary Heisenberg group
$E_m=\F_2\times\F_2^m\times\F_2^m$, with multiplication
\[
 (z,u,v)(z',u',v')=(z+z'+u\cdot v',\ u+u',\ v+v').
\]
Its centre and derived subgroup both equal
$\F_2\times\{0\}\times\{0\}$, and the commutator induces the
nondegenerate alternating form
$b((u,v),(u',v'))=u\cdot v'+u'\cdot v$.
For $t$ pairwise nonorthogonal vectors, the Gram matrix is
$J_t+I_t$. Over $\F_2$ its kernel consists of the constant
vectors $c\mathbf 1$ satisfying $c(t-1)=0$; hence its rank is
$t$ for even $t$ and $t-1$ for odd $t$. As this rank is at most
$2m$, we obtain $t\le2m+1$. For the converse, the alternating
matrix $J_{2m+1}+I_{2m+1}$ has rank $2m$ and radical spanned
by $\mathbf 1=(1,\ldots,1)$. Its quotient by the radical is a
nondegenerate alternating space of dimension $2m$. The images
of the standard basis vectors are pairwise distinct (since
$e_i+e_j\ne\mathbf 1$ for $2m+1\ge3$) and pairwise
nonorthogonal. All nondegenerate alternating forms of dimension
$2m$ over $\F_2$ are isometric, giving
$\om(E_m)=2m+1$.

Identify $V$ with $\F_{2^m}^2$ and take
$b((x,y),(x',y'))=\operatorname{Tr}_{\F_{2^m}/\F_2}
(xy'-x'y)$. This form is nondegenerate. Each of the
$2^m+1$ one-dimensional $\F_{2^m}$-subspaces is an
$m$-dimensional totally isotropic $\F_2$-subspace, and these
subspaces partition $V\setminus\{0\}$. They therefore give
an isotropic cover of size $2^m+1$.
Conversely, every totally isotropic subspace has dimension at
most $m$ and contains at most $2^m-1$ nonzero vectors. Counting
all $2^{2m}-1$ nonzero vectors shows that any isotropic cover
has at least
\[
 \left\lceil\frac{2^{2m}-1}{2^m-1}\right\rceil=2^m+1
\]
members. Finally, a totally isotropic subspace has order at
most $2^m$, and the spread exhibits one of that order, so its
least possible index in $V$ is $2^m$. Taking preimages in
$E_m$ establishes $a(E_m)=2^m+1$ and
$\indab(E_m)=2^m$.
\end{proof}

Taking $m=\lfloor(n-1)/2\rfloor$ gives $h(n)\ge L_n+1$ for $n\ge3$.
The same calculation applies to every nondegenerate scalar form:
in dimension $2s$ over $\F_p$, its isotropic covering number is
$p^s+1$ and its least isotropic index is $p^s$.
For $p=2$, its clique number is exactly $2s+1$.
We now seek upper bounds that distinguish other commutator maps
from this model.

\subsection{Scalarization and isotropic index}

A \emph{scalarization} is $g=\lambda\beta$ for a homomorphism
$\lambda:W\to\F_p$. Its rank is $\dim_{\F_p}(V/\rad g)$;
this quotient is elementary abelian since $pV\le\rad g$.
For scalar forms, $\perp_g$ denotes orthogonal direct sum and
$X^{\perp g}$ the orthogonal complement of $X$.
All vector-space dimensions and ranks are over $\F_p$ in the
$p$-primary setting.
Write $N_R(v)=\om(v+R)$, also denoted $N_R(q)$ for $q\in V/R$.
We shall use
\begin{equation}\label{eq:budget}
 \tau(\beta)\le\log_p[V:R]+\tau(\beta|_R),\qquad
 a(\beta)\le\om(\beta)p^{\tau(\beta)}.
\end{equation}
The index inequality is immediate. For the covering inequality, choose an isotropic
subgroup $A$ of minimum index. In $v+A$, two elements $v+a,v+b$
are orthogonal exactly when $a-b\in C_A(v)$.
Representatives of $A/C_A(v)$ therefore give both a clique and a
cover by the isotropic subgroups $\langle v+a,C_A(v)\rangle$.
More generally, let $R\le V$ and let $A\le R$ be isotropic of
index $p^t=p^{\tau(\beta|_R)}$. Partition $v+R$ into its $p^t$
cosets $x+A$. In each one, orthogonality is the equivalence relation
\[
 x+a\perp x+b\quad\Longleftrightarrow\quad
 a-b\in C_A(x),
\]
because $\beta(A,A)=0$. Its classes are the cosets of $C_A(x)$;
representatives of different classes form a nonorthogonal set in
$v+R$, so their number is at most $N_R(v)$. Each class is contained
in the isotropic subgroup $\langle x+a,C_A(x)\rangle$, for a
suitable representative $a$. Consequently $v+R$ is covered by at
most $p^tN_R(v)$ isotropic subgroups. The same subgroups cover all
its nonzero scalar multiples in $V/R$: multiplication by any integer
coprime to $p$ permutes $R$ and preserves each subgroup.

A scalarization extracts an ordinary alternating form from a
commutator map with values in a group. Distinct cosets of its
radical can then be combined into a single nonorthogonal set
whenever their scalar pairings are nonzero. This observation
underlies the index bound below.

\begin{lemma}[Cliques in symplectic spaces]\label{lem:scalar}
A symplectic space of dimension $2s$ over $\F_p$, with $s\ge1$,
contains a clique of $ps+1$ elements, exactly one of which lies
in any prescribed hyperplane. If $p=3$, it also contains a clique
of size at least $10s/3$.
\end{lemma}
\begin{proof}
Representatives of the $p+1$ projective points in a symplectic plane
form a clique. In an orthogonal direct sum, cliques $C_1,C_2$ and
an element $v\in C_2$ give the clique
$(C_1\times\{v\})\cup(\{0\}\times(C_2\setminus\{v\}))$,
of size $|C_1|+|C_2|-1$. Each additional plane therefore contributes
$p$ elements. For a prescribed hyperplane, choose the last plane
so that its orthogonal complement lies in the hyperplane, and
choose $v$ outside it. Exactly one of the remaining projective
points of the last plane lies in the hyperplane.

For the stronger ternary bound, use the classical skew-Hadamard doubling
construction (cf. \cite[Chapter~14]{Hall1986}), written with
$D_r=H_r-I_r$.
Define integer matrices recursively for powers of two:
\[
 D_1=0,\qquad
 D_{2r}=\begin{pmatrix}D_r&D_r+I_r\\D_r-I_r&-D_r\end{pmatrix}.
\]
Block multiplication gives $D_r^T=-D_r$ and
$D_r^2=-(r-1)I_r$; every off-diagonal entry is $\pm1$.
Over $\F_3$, we have $D_{16}^2=0$, so $\rank D_{16}\le8$.
Quotienting the corresponding alternating form by its radical
gives sixteen pairwise nonorthogonal elements in a symplectic
space of dimension at most eight. Adding hyperbolic planes if
necessary embeds this configuration in dimension eight.
Write $s=4k+r$, where $0\le r\le3$. Combining $k$ such blocks
and $r$ planes gives $15k+3r+1\ge10s/3$, since the difference
is $(5k-r+3)/3\ge0$.
\end{proof}

The orthogonal-sum clique construction uses only the vanishing of
cross-pairings between the summands. It therefore applies without
change to alternating maps taking values in any abelian group.

The orthogonal-sum construction and the bound $ps+1$ also appear in
\cite[Section~4.2 and Corollary~4.4]{SarkarYoder2024}.
The reduction of the skew-Hadamard matrix modulo three gives the
ternary estimate used here.

\begin{lemma}[Scalarization]\label{lem:scalarisation}
Let $\lambda:W\to\F_p$ be a homomorphism such that
$g=\lambda\beta\ne0$. Suppose that $g$ has radical $R$ and
$[V:R]=p^{2s}$. Put
$t=\tau(\beta|_R)$, $n=\om(\beta)$,
$c=ps+1$ and $B=(p^{2s}-1)/((p-1)c)$. Then
\begin{equation}\label{eq:average}
 n\ge\sum_{q\in C}N_R(q),\qquad
 a(\beta)\le a(\beta|_R)+p^tBn,
\end{equation}
where $C$ is any clique for the induced scalar form on $V/R$.
If $R$ is isotropic and $\beta(R,V)\ne0$, then $n\ge p^2s+1$.

Suppose instead that $\beta\ne0$, its values generate $W$,
$\rad\beta=0$, and $\lambda\beta$ is nondegenerate for every
nonzero homomorphism $\lambda:W\to\F_p$. Then $V\cong\F_p^{2s}$
for some integer $s\ge1$, and $W$
is a vector space.
If $\dim W=1$, then $\tau(\beta)=s$ and $a(\beta)=p^s+1$;
otherwise $s\ge2$ and
\begin{equation}\label{eq:nonsingclique}
 n\ge p^2(s-1)+1,\qquad a(\beta)\le\frac{p^{2s}-1}{p-1}.
\end{equation}
\end{lemma}
\begin{proof}
\emph{Averaging over scalar directions.}
Since $pV\le R$, the quotient $V/R$ is a vector space over
$\F_p$, and the form induced by $g$ is symplectic.
Cliques chosen in distinct cosets indexed by $C$ remain pairwise
nonorthogonal under $g$, proving the first inequality. The local
covering construction uses at most $p^tN_R(v)$ subgroups for each
projective direction: the same subgroups cover all nonzero scalar
multiples of $v+R$, since multiplication by an integer coprime to
$p$ preserves every subgroup and maps $R$ bijectively to itself.
Likewise $N_R(v)$ depends only on the projective direction of $v+R$;
write it as $N(\ell)$. A clique of size $c$ meets $c$ distinct
directions. The symplectic group acts transitively on nonzero vectors (and hence
on their projective directions). Thus each projective direction
occurs equally often in the symplectic translates of $C$.
Averaging over these translates gives
\[
 \frac{c}{m}\sum_\ell N(\ell)\le n,
 \qquad m=\frac{p^{2s}-1}{p-1}.
\]
Together with a cover of $R$, this proves \eqref{eq:average}.
Only the scalar form, not $\beta$, is required to be invariant
under the symplectic group.

\emph{The isotropic-radical gain.}
If $R$ is isotropic, $C_V(R)/R$ is a proper subspace of $V/R$.
Take a hyperplane containing it and apply Lemma~\ref{lem:scalar}.
The $ps$ cosets outside that hyperplane have nonzero interaction
with $R$, so each contains a clique of at least $p$ elements.
The remaining coset contributes one, giving $n\ge p^2s+1$.

\emph{Nondegenerate scalarizations.}
In the nondegenerate case, a nondegenerate scalarization forces
$pV=0$, and hence $pW=0$. If $\dim W=1$, the scalar calculation following
Lemma~\ref{lem:extraspecial} gives the result.
If $\dim W\ge2$, then $s\ge2$, since the space of alternating
forms in dimension two is one-dimensional. Also $\beta(z,V)=W$ for
every $z\ne0$: a linear functional annihilating a proper image
would yield a degenerate scalarization. On $z^{\perp g}$, the
scalar radical is $\langle z\rangle$, and surjectivity gives
$\beta(z,z^{\perp g})=\ker\lambda$, of dimension $\dim W-1>0$.
Applying the isotropic-radical
case just proved gives $n\ge p^2(s-1)+1$. All one-dimensional
subspaces are isotropic, giving the covering bound.
\end{proof}

\begin{proposition}[Isotropic index]\label{prop:isotropic-index}
Every alternating map between finite abelian $p$-groups satisfies
\[
 \om(\beta)\ge
 \begin{cases}
 2\tau(\beta)+1,&p=2,\\
 (10/3)\tau(\beta),&p=3,\\
 p\tau(\beta),&p\ge5.
 \end{cases}
\]
\end{proposition}
\begin{proof}
Write $n=\om(\beta)$ and induct on $|V|$. The proof distinguishes
degenerate scalarizations, where we combine cliques from several
cosets of a radical, from nondegenerate scalarizations, where
Lemma~\ref{lem:scalarisation} gives a direct bound.

Quotienting by the radical preserves $n$ and $\tau$, so we may
assume $\rad\beta=0$ and replace $W$ by the subgroup generated
by the values. The zero map has $\tau=0$ and $n=1$.
Otherwise $W\ne0$, and some homomorphism
$\lambda:W\to\F_p$ has $g=\lambda\beta\ne0$.

\emph{Case 1: a degenerate scalarization.}
Suppose first that $g$ is degenerate. Let $R=\rad g$, with
$0<R<V$, $[V:R]=p^{2s}$, and $t=\tau(\beta|_R)$. An isotropic
subgroup in $R$ is also isotropic in $V$; hence
\begin{equation}\label{eq:index-budget-expanded}
 \tau(\beta)\le 2s+t.
\end{equation}
\emph{The isotropic radical ($t=0$).}
If $t=0$, then $R$ is isotropic. Since $\rad\beta=0$, one has
$\beta(R,V)\ne0$, so Lemma~\ref{lem:scalarisation} gives
$n\ge p^2s+1$. This is at least $4s+1$ for $p=2$,
$(20/3)s$ for $p=3$, and $2ps$ for $p\ge5$; the desired
assertion follows from \eqref{eq:index-budget-expanded}.

\emph{A clique bound in every coset ($t\ge1$).}
Assume $t\ge1$. For $j\ge1$ write
\[
 L_2(j)=2j+1,\qquad L_3(j)=\left\lceil\frac{10j}{3}\right\rceil,
 \qquad L_p(j)=pj\ (p\ge5).
\]
We first prove, for $1\le j\le t$, that each coset $v+R$ contains
a nonorthogonal set of at least
\begin{equation}\label{eq:greedy}
 N_R(v)\ge\min\{p^j,L_p(t-j+1)\}.
\end{equation}
Choose a maximal nonorthogonal family in $v+R$. For each selected
$x$, the elements of $v+R$ orthogonal to $x$ are exactly
$x+C_R(x)$. If all these subgroups $C_R(x)$ have index at least
$p^j$ in $R$, fewer than $p^j$ orthogonality sets cannot cover
$v+R$, so the family has at least $p^j$ members. Otherwise there
is a point $x\in v+R$ such that, with $K=C_R(x)$,
$[R:K]\le p^{j-1}$. If an isotropic subgroup of $K$ has index
$p^d$ in $K$, it has index $[R:K]p^d$ in $R$, so
\[
 \tau(\beta|_K)\ge t-\log_p[R:K]\ge t-j+1.
\]
Since $K\subseteq R\subsetneq V$, induction applies to the
restricted map $\beta|_K$: it has a nonorthogonal set of size
at least $L_p(t-j+1)$. Translating this set by $x$ does not alter
its pairwise pairings because $\beta(x,K)=0$. This proves
\eqref{eq:greedy}.

\emph{The exceptional binary subcase.}
For $p=2,t=1$, a sharper assertion is needed: \emph{every coset
$v+R$ has a nonorthogonal triple}. Indeed, $\beta|_R\ne0$;
choose $a,b\in R$ with $\beta(a,b)\ne0$. The three points
$a,b,a+b$ form a nonorthogonal triple. If any $x\in v+R$
centralizes $R$, its translate $x+\{a,b,a+b\}$ is such a triple
in $v+R$.

Suppose, to the contrary, that $v+R$ has no nonorthogonal triple.
Then no point of this coset centralizes $R$, so every
$C_R(x)$ is a proper subgroup and has index at least two.
A maximal nonorthogonal family in $v+R$ has at least two members
by the covering argument above, and exactly two under our
assumption; denote them by $x_1,x_2$. Their orthogonality sets
\[
 S_i=x_i+C_R(x_i)\qquad(i=1,2)
\]
cover $v+R$, and each has size at most $|R|/2$. It follows that
they are disjoint, each has size $|R|/2$, and
$K_i=C_R(x_i)$ has index two in $R$. Each $S_i$ is pairwise
orthogonal: if $y,z\in S_i$ were nonorthogonal, then both would
be nonorthogonal to $x_{3-i}$, because $S_i$ is disjoint from
$S_{3-i}$; hence $\{y,z,x_{3-i}\}$ would be a triple.

If $K_1\ne K_2$, the quotient map
$R\to R/K_1\times R/K_2\cong\F_2^2$ is onto, so any coset
of $K_1$ intersects any coset of $K_2$. This contradicts
$S_1\cap S_2=\varnothing$. Thus $K_1=K_2=K$.
The isotropy of each $S_i=x_i+K$ implies $\beta(K,K)=0$.
Each $x_i$ centralizes $K$ by definition, and the two sets
$S_1,S_2$ partition $v+R$. Therefore every point of $v+R$
centralizes $K$. Taking differences shows that
$\beta(R,K)=0$, i.e. $K\le\rad(\beta|_R)$.
But $R/K$ has order two and is cyclic; an alternating map
on a cyclic group vanishes, giving $\beta(R,R)=0$, contrary
to $t=1$. Thus $N_R(v)\ge3$ for $p=2,t=1$.

\emph{Choosing the greedy parameter.}
We now select $j$ in \eqref{eq:greedy} for each prime and each
remaining value of $t$. For $p=2,t\ge2$, set
$j=\lfloor\log_2t\rfloor+1\le t$. Then $2^j\ge t+1$ and
$2\lfloor\log_2t\rfloor\le t$, so
$L_2(t-j+1)=2(t-j+1)+1\ge t+1$. Thus $N_R(v)\ge t+1$.
For $p=3,t=1$, take $j=1$, obtaining $N_R(v)\ge3$.
For $p=3,t=2$, take $j=2$, obtaining
$N_R(v)\ge\min\{9,L_3(1)\}=4$.
For $p=3,t\ge3$, put $j=\lfloor2t/5\rfloor+1\ge2$.
Then $j\le t$, $2t\le5j-1\le3^j$ (the second inequality
holds at $j=2$ and by induction), and
\[
 \frac{10}{3}(t-j+1)\ge2t.
\]
Thus $N_R(v)\ge2t$.
For $p\ge5,t=1$ use $j=1$, obtaining $N_R(v)\ge p$.
For $p\ge5,t\ge2$, choose $j=\lceil t/2\rceil\le t$.
The elementary inequalities $p^{j-1}\ge j\ge t/2$ and
$t-j+1\ge t/2$ give $N_R(v)\ge pt/2$.

\emph{Combining the cosets.}
Take a clique of $ps+1$ cosets in $V/R$, as supplied by
Lemma~\ref{lem:scalar}. Every two cosets in this family are
nonorthogonal under $g$, so nonorthogonal families in these
cosets concatenate to a single nonorthogonal family in $V$.
In summary,
\[
\begin{array}{c|cc}
 &t=1&t\ge2\\ \hline
 p=2&3&t+1\\
 p=3&3&4\ (t=2),\quad 2t\ (t\ge3)\\
 p\ge5&p&pt/2
\end{array}
\]
are valid lower bounds for each $N_R(v)$. In conjunction with
\eqref{eq:index-budget-expanded}, the required comparisons for
$t\ge2$ are
\[
\begin{split}
 (t+1)(2s+1)&\ge 4s+2t+1,\\
 2t(3s+1)&\ge\tfrac{10}{3}(2s+t),\\
 (pt/2)(ps+1)&\ge p(2s+t).
\end{split}
\]
For $p=3,t=2$ the second comparison also holds with the actual
bound $N_R(v)\ge4=2t$. The first difference is
$2s(t-1)-t\ge0$; the second difference is
$[t(18s-4)-20s]/3\ge0$ for $t\ge2,s\ge1$;
the third is equivalent to $s(pt-4)\ge t$, valid for
$p\ge5,t\ge2,s\ge1$. For $t=1$ the comparisons are
$3(2s+1)\ge4s+3$, $3(3s+1)\ge(10/3)(2s+1)$, and
$p(ps+1)\ge p(2s+1)$.

\emph{Case 2: all nonzero scalarizations are nondegenerate.}
It remains to treat this case. Lemma~\ref{lem:scalarisation} makes $V$ and $W$
vector spaces with $\dim V=2s$. If $\dim W=1$, then
$\tau(\beta)=s$, and Lemma~\ref{lem:scalar} supplies the clique
bound. Otherwise $s\ge2$ and
$n\ge p^2(s-1)+1$. For odd $p$, an isotropic line gives
$\tau(\beta)\le2s-1$, while
\[
 p^2(s-1)+1\ge
 \begin{cases}
 (10/3)(2s-1),&p=3,\\
 p(2s-1),&p\ge5.
 \end{cases}
\]
For $p=2$, if an isotropic plane exists, then
$\tau(\beta)\le2s-2$ and $n\ge4s-3\ge2\tau(\beta)+1$.
If no isotropic plane exists, every pair of distinct nonzero
vectors has nonzero pairing: all $2^{2s}-1$ nonzero vectors
form a clique, and $\tau(\beta)=2s-1$. Since $s\ge2$,
$2^{2s}-1\ge4s-1=2\tau(\beta)+1$.
\end{proof}

The index estimate now passes to class-two groups.

\begin{corollary}[Class-two index bounds]\label{cor:class2}
Every finite group $G$ of class at most two, with $n=\om(G)$, satisfies
\begin{equation}\label{eq:class2out}
 \indab(G)\le2^{n/2}.
\end{equation}
If $G$ has at least two nonabelian Sylow subgroups, then
$\indab(G)\le2^{7n/24}$.
\end{corollary}
\begin{proof}
For a Sylow $p$-subgroup $P$, Proposition~\ref{prop:isotropic-index}
gives $\indab(P)\le2^{\om(P)/2}$, using
$2\log_2 3<10/3$ and $p^2\le2^p$ for $p\ge5$.
Since $G$ is finite nilpotent, it is the direct product of its Sylow
subgroups. Discard abelian factors; if none remain the result is immediate.
Abelian indices multiply as upper bounds
and clique sizes multiply as lower bounds. For factors with clique
sizes at least three, their sum is at most their product, proving
\eqref{eq:class2out}.
If at least two factors remain, split $G$ into two nonabelian factors
with clique sizes $a\ge4$ and $b\ge3$: one factor contains a
nonabelian odd Sylow subgroup, whose clique number is at least four
by Proposition~\ref{prop:isotropic-index}. Then
\[
 \log_2\indab(G)\le(a+b)/2\le7ab/24\le7n/24.
\]
\end{proof}

\section{Reduction to binary alternating maps}\label{sec:tools}

For the upper bound, we first exclude groups outside the
binary class-two model.

\begin{proposition}[Reduction to the binary case]\label{prop:binary-reduction}
Let $n\ge64$ be even and $G$ a finite group with $\om(G)\le n$.
If $G$ is not of the form $P\times A$, where $P$ is a class-two
$2$-group and $A$ is abelian of odd order, then $a(G)<L_n$.
\end{proposition}

We first eliminate groups outside class two by centralizer descent,
then eliminate nonabelian odd Sylow factors. The argument does not
assume nilpotency at the outset.

\subsection{Centralizer descent}

The descent rests on two elementary mechanisms: a large clique
controls centralizer indices inside a coset, while a clique outside
a subgroup forces either a small abelian index or a small
centralizer. We record both mechanisms before the contraction step.

\begin{lemma}[Coset bounds]\label{lem:covergroup}
For $B\le G$ and $S=yB$, some $x\in S$ satisfies
$[B:C_B(x)]\le\om(S)$. Also $a(G)\le\om(G)\indab(G)$.
\end{lemma}
\begin{proof}
If $x_1,\ldots,x_k$ is a maximum clique in $S$, the sets
$x_iC_B(x_i)$ cover $S$, so one has size at least $|B|/k$.
For the second assertion take an abelian subgroup $A$ of minimum index.
In $gA$, two elements $ga,gb$ commute exactly when
$ab^{-1}\in D=C_A(g)$. Representatives of $A/D$ give both a
clique and a cover by the abelian subgroups $\langle ga,D\rangle$.
Thus each coset costs at most $\om(G)$ subgroups.
\end{proof}

\begin{lemma}[Rigidity on the complement of a subgroup]\label{lem:rigidity}
Let $D<G$ satisfy $\om(G\setminus D)\le n-1$.
Either $\indab(G)\le n-1$, or some $x\notin D$ satisfies
$[G:C_G(x)]\le n-2$.
\end{lemma}
\begin{proof}
Put $S=G\setminus D$ and $d=[G:D]$, and suppose that
$[G:C_G(x)]\ge n-1$ for every $x\in S$.
For $C=C_G(x)$, we have $q=[C:C\cap D]\le d$, whence
\[
 |C\cap S|=|C|(1-1/q)
 \le\frac{|G|(1-1/d)}{n-1}=\frac{|S|}{n-1}.
\]
A maximal clique $x_1,\ldots,x_k$ in $S$ gives a cover by the
sets $C_G(x_j)\cap S$. Since $k\le n-1$, equality holds throughout:
$k=n-1$, the sets are disjoint, and each $C_G(x_j)$ has index $n-1$.
Each part is a commuting set. Indeed, if $y,z$ in the $j$-th
part did not commute, disjointness of the cover would make both
$y$ and $z$ noncommuting with every $x_i$ for $i\ne j$.
Replacing $x_j$ by $y,z$ would then give a clique of size $n$
in $S$, a contradiction. Moreover, $C_G(x_j)\cap S$ generates
$C_G(x_j)$: if $s$ belongs to
this part and $a\in C_G(x_j)\cap D$, then $a=s^{-1}(sa)$ with both
$s$ and $sa$ in the part. Thus $C_G(x_j)$ is abelian and
$\indab(G)\le n-1$.
\end{proof}

\begin{lemma}[Centralizer extension and small quotients]\label{lem:small-quotients}
If $x\notin Z(G)$, then $\om(C_G(x))\le\om(G)-2$.
If a nonabelian group $T$ has clique number three, then
$T/Z(T)\cong C_2^2$ and $\indab(T)=2$.
If its clique number is four, then
$T/Z(T)\cong S_3$ or $C_3^2$, and $\indab(T)\le3$.
Here $S_3$ is the symmetric group on three letters.
\end{lemma}
\begin{proof}
More generally, if $B\le G$ and $g$ does not centralize $Z(B)$, then
$\om(G)\ge\om(B)+[Z(B):C_{Z(B)}(g)]$.
Choose $t\in Z(B)$ not commuting with $g$. In a clique of $B$, replace
any element commuting with $g$ by its product with $t$.
Pairwise commutators are unchanged, and the resulting clique commutes
with no element of $gZ(B)$. This coset contains a clique of size
$[Z(B):C_{Z(B)}(g)]$ by the coset calculation in
Lemma~\ref{lem:covergroup}. Taking $B=C_G(x)$ proves the first assertion.

If $r=\om(T)\le4$, every proper centralizer is abelian, since otherwise
the first assertion gives a clique of size at least five.
Distinct proper centralizers $C_1,\ldots,C_r$ intersect in $Z(T)$;
the sets $C_j\setminus Z(T)$ partition $T\setminus Z(T)$, and elements
in different parts do not commute. There are exactly $r$ such
centralizers: choosing one element from each part gives a clique,
whereas no clique contains two elements of the same part.
Put $M=[T:Z(T)]$ and $m_j=[C_j:Z(T)]$, ordered decreasingly. Then
\[
 M=1+\sum_j(m_j-1),\qquad m_1m_2\le M\le m_1+(r-1)(m_2-1),
\]
so $m_1\le r-1$. For $r=3$, all $m_j=2$ and $M=4$.
For $r=4$, each $m_j$ is two or three. If both values occur,
$6\mid M$ and $6\le M\le8$, hence $M=6$.
If all $m_j=2$, then $M=5$, contradicting $2\mid M$;
if all $m_j=3$, then $M=9$.
A nonabelian group cannot have cyclic central quotient, giving the
stated quotients. The largest $C_j$ gives the asserted abelian index.
\end{proof}

The following lemma isolates the product mechanism used in
\cite[Lemma~5.7]{Lecomte2026}.

\begin{lemma}[Amplification]\label{lem:amplification}
Let $H,L\le G$ satisfy $[H,L]=1$. If $T\subseteq L$ has $[t,t']\notin H'$ for
distinct $t,t'\in T$, then $\om(G)\ge|T|\om(H)$. In particular,
\[
 \om(G)\ge\om\bigl(L/(L\cap H)\bigr)\om(H).
\]
If $X\le C_G(H)$ is abelian and $q=[HX:H]$, then
$\indab(G)\le[G:H]\indab(H)/q$.
\end{lemma}
\begin{proof}
For a clique $E\subseteq H$, the identity
$[te,t'e']=[t,t'][e,e']$ shows that all $|T||E|$ indexed products
are pairwise noncommuting and hence distinct: the two commutator
factors cannot cancel when $t\ne t'$, and when $t=t'$ the second
factor is nontrivial for $e\ne e'$.
For the quotient assertion, $L\cap H\le Z(L)$, and lifted clique
commutators lie outside $L\cap H$, hence outside $H'$.
For the index bound, if $A\le H$ is abelian of minimum index, then
$AX$ is abelian and $[AX:A]\ge[X:X\cap H]=q$.
\end{proof}

\begin{proposition}[Structural contraction]\label{prop:descent}
Let $G'\not\le Z(G)$ and $n\ge\om(G)$.
Either $\indab(G)\le n-1$, or there is $H\le G$ satisfying
one of the two contractions
\[
\begin{array}{c|cc}
 &\om(H)&\indab(G)/\indab(H)\\ \hline
 A&\le\lfloor n/3\rfloor&\le(n-2)^2/2\\
 B&\le\lfloor n/5\rfloor&\le(n-2)^2(n-3)/4.
\end{array}
\]
\end{proposition}
\begin{proof}
We construct two noncommuting elements with small centralizers,
then split according to whether their commutator belongs to the
derived subgroup of the common centralizer. The latter case
requires a third element and provides the stronger clique reduction.

Assume $\indab(G)>n-1$. Write $Z_2(G)/Z(G)=Z(G/Z(G))$.
For $t\notin Z_2(G)$, the subgroup
$D_t=\{x\in G:[t,x]\in Z(G)\}$ is proper and contains $Z_2(G)$.
Every clique in $G\setminus D_t$ extends by $t$, so its size is at
most $n-1$. As $G'\not\le Z(G)$, there is $w\notin Z_2(G)$.
Apply Lemma~\ref{lem:rigidity} to $D_w$ to obtain $u\notin D_w$
whose centralizer has index at most $n-2$. Because $D_w$ contains
$Z_2(G)$, we may apply the lemma again, now to $D_u$, to obtain
$v\notin D_u$ with the same index bound. In particular
\[
 c=[u,v]\notin Z(G),\qquad
 [G:C_G(u)],[G:C_G(v)]\le n-2.
\]
Put $K=C_G(u,v)$ and $U=\langle u,v\rangle$; thus $[G:K]\le(n-2)^2$
and $[U,K]=1$. The two alternatives $c\notin K'$ and $c\in K'$ are exhaustive.
The first gives a three-point clique multiplier and an index saving
of two; the second requires a third generator but gives a five-point
multiplier and an index saving of four.

\emph{Branch $A$: $c\notin K'$.}
The triple $u,v,uv$ has all pairwise commutators outside $K'$:
these are conjugates of $c^{\pm1}$. Since $u$ and $v$ centralize
$K$ pointwise, they normalize $K'$ and consequently none of
these conjugates can lie in $K'$.
Lemma~\ref{lem:amplification} thus gives $3\om(K)\le n$.
Moreover $u\notin K$, so adjoining $\langle u\rangle$ to a
minimum-index abelian subgroup of $K$ saves at least a factor of
$2$. Consequently
\[
 \om(K)\le\lfloor n/3\rfloor,\qquad
 \indab(G)\le\frac{(n-2)^2}{2}\indab(K),
\]
which is branch $A$.

\emph{Branch $B$: the commutator lies in $K'$.}
Suppose $c\in K'$. Since $U$ centralizes $K$, its commutator
$c=[u,v]$ centralizes $K$ and lies in $K$; thus $c\in Z(K)$.
Moreover $u,v$ centralize $c$, so $c$ commutes with all of $U$.
The element $c$ is noncentral in $G$, hence there exists
$z_0\notin C_G(c)$. For each $t\in\{u,v,uv\}$ choose $t$ if
$[t,z_0]\ne1$, and choose $tc$ otherwise. The selected element
does not commute with $z_0$, since $[c,z_0]\ne1$; the three
selected elements are pairwise noncommuting, because $c$ commutes
with $u,v$. They all centralize $K$ and are noncommuting with
every element of $z_0K$. Hence every clique in $z_0K$ extends
by these three points, so
\[
 \om(z_0K)\le n-3.
\]
The coset bound in Lemma~\ref{lem:covergroup} gives
$z\in z_0K$ with $[K:C_K(z)]\le n-3$.
As $c\in Z(K)$, $[c,z]\ne1$. Set
\[
 L=\langle u,v,z\rangle,\qquad
 H=C_G(L)=K\cap C_G(z),\qquad
 N=L\cap H=Z(L).
\]
The index estimate is
\begin{equation}\label{eq:B-index-expanded}
 [G:H]\le(n-2)^2(n-3).
\end{equation}

\emph{Excluding small central quotients.}
We claim that $\om(L/N)\ge5$. Put $Q=L/N$. If $Q$ were
abelian, then $[u,v]=c\in N$, contradicting $[c,z]\ne1$;
thus $Q$ is nonabelian. Suppose that $\om(Q)\le4$.
Lemma~\ref{lem:small-quotients} implies
\[
 Q/Z(Q)\cong C_2^2,\quad C_3^2,\quad\text{or }S_3.
\]
We rule out all three possibilities.

If $Q/Z(Q)\cong S_3$, let $\bar u,\bar v$ be the images of
$u,v$. Since $u$ and $v$ centralize $c$, both images centralize
$[\bar u,\bar v]$. A nonidentity commutator in $S_3$ is a
$3$-cycle; the centralizer of a $3$-cycle is the cyclic group of
order three. Hence, if the commutator is nonidentity, $\bar u$
and $\bar v$ lie in the same cyclic subgroup. If it is the
identity, they generate an abelian (therefore cyclic) subgroup
of $S_3$ as well. In either case, $\bar u$ and $\bar v$ lie in
a common cyclic subgroup of $Q/Z(Q)$. Their lifts to $Q$
therefore commute: writing them as powers of a fixed lift times
central elements shows $[uN,vN]=1$. This says $c\in N$, a
contradiction.

If $Q/Z(Q)\cong C_p^2$ for $p\in\{2,3\}$, the images of
$u,v$ must be linearly independent. Indeed, dependence would
put them in a single cyclic subgroup of $Q/Z(Q)$ and hence make
their lifts commute in $Q$, forcing $c\in N$. Their images
therefore span $Q/Z(Q)$. Consequently, there are integers $a,b$
such that $w=zu^av^b$ has image in $Z(Q)$, or equivalently
$w\in Z_2(L)$, where $Z_2(L)$ is the preimage of $Z(Q)$ in $L$.
An element $w\in Z_2(L)$ acts on $L$ by a central automorphism:
for every $x\in L$, $x^w=xz_x$ with $z_x\in Z(L)$.
Thus $[x,y]^w=[xz_x,yz_y]=[x,y]$, showing that $w$
centralizes $L'$, in particular $c=[u,v]$.
But $u$ and $v$ centralize $c$, so $z$ centralizes $c$,
contrary to $[c,z]\ne1$. The claim follows.

\emph{Clique amplification and index saving.}
By Lemma~\ref{lem:amplification},
\begin{equation}\label{eq:B-clique-expanded}
 5\om(H)\le\om(G)\le n.
\end{equation}
To save the additional index factor, set
$X=\langle u,c\rangle\le C_G(H)$. This group is abelian,
since $u$ centralizes $c$. The images $uN,cN$ in $Q$ are both
nonidentity: $u\in N$ would force $[u,v]=1$, while $c\in N$
would contradict $[c,z]\ne1$. Their generated subgroup is
abelian. If it had order less than four, it would be cyclic of
order two or three; in particular $uN$ would be a power of
$cN$. Because $c$ centralizes $v$, this would imply
$[uN,vN]=1$, a contradiction. Hence
\[
 [HX:H]=|X/(X\cap H)|=|XN/N|\ge4.
\]
Applying the index inequality in Lemma~\ref{lem:amplification}
and \eqref{eq:B-index-expanded} now gives
$\indab(G)/\indab(H)\le(n-2)^2(n-3)/4$, completing branch $B$.
\end{proof}

Both contractions reduce the clique parameter by a fixed factor.
The following uniform estimate is proved directly from these two
branches. It avoids accumulating a polynomial factor through the
descent.

\begin{lemma}[Uniform exponential index]\label{prop:uniform-index}
Every finite group $G$ satisfies
\begin{equation}\label{eq:uniform-eight}
 \indab(G)\le 8\,2^{\om(G)/2}.
\end{equation}
\end{lemma}
\begin{proof}
Write $n=\om(G)$. The class-two estimate
$\indab(G)\le2^{n/2}$ and the alternative $\indab(G)\le n-1$
are already below the right-hand side.
For a finite initial verification, set $U_j=2^{j/2}$ for
$0\le j\le4$, and, for $5\le j\le31$, define
\[
 U_j=\max\left\{j-1,2^{j/2},
 \frac{(j-2)^2}{2}U_{\lfloor j/3\rfloor},
 \frac{(j-2)^2(j-3)}4U_{\lfloor j/5\rfloor}\right\}.
\]
Proposition~\ref{prop:descent}, induction on $j$, and the
small-quotient cases in Lemma~\ref{lem:small-quotients} imply
$\indab(G)\le U_n$ when $n\le31$. Direct substitution yields
\[
\begin{array}{c|cccc}
 j\text{ in}&[0,8]&[9,16]&[17,24]&[25,31]\\ \hline
 \max U_j/2^{j/2}&45\sqrt2/16&507/64&1575/256&52371/16384
\end{array}
\]
(the maxima occur at $j=8,15,17,25$, respectively).
Each entry is less than $8$.

For $n\ge32$ we proceed by induction. The two contraction
multipliers, normalized by $2^{n/2}$, satisfy
\[
\begin{aligned}
 \frac{(n-2)^2}{2}2^{(\lfloor n/3\rfloor-n)/2}
 &\le \frac{(n-2)^2}{2}2^{-n/3},\\
 \frac{(n-2)^2(n-3)}4 2^{(\lfloor n/5\rfloor-n)/2}
 &\le \frac{(n-2)^2(n-3)}4 2^{-2n/5}.
\end{aligned}
\]
Both right-hand sides decrease for real $n\ge32$:
their logarithmic derivatives are, respectively,
$2/(n-2)-(\log2)/3$ and
$2/(n-2)+1/(n-3)-2(\log2)/5$, which are negative since
$\log2>2/3$. At $n=32$ they are less than one
(the first is less than $450/2^{10}$, and the second is
$6525/2^{64/5}<6525/(4096\cdot8/5)<1$,
since $2^{4/5}>8/5$).
Thus each contraction preserves the bound
$8\,2^{n/2}$, proving \eqref{eq:uniform-eight}.
\end{proof}

One descent step reduces the clique parameter to at most $n/3$.

\begin{lemma}[Separation outside class two]\label{prop:nonclass2-gap}
If $G'\not\le Z(G)$ and $\om(G)\le n$, then
$a(G)<L_n$ for every even integer $n\ge64$.
\end{lemma}
\begin{proof}
By Proposition~\ref{prop:descent}, Lemma~\ref{prop:uniform-index},
and $a(G)\le n\indab(G)$, we have either $a(G)\le n(n-1)$,
or
\begin{equation}\label{eq:two-descent-estimates}
 a(G)\le
 \begin{cases}
 4n(n-2)^2\,2^{\lfloor n/3\rfloor/2},&\text{branch $A$},\\
 2n(n-2)^2(n-3)\,2^{\lfloor n/5\rfloor/2},&\text{branch $B$}.
 \end{cases}
\end{equation}
Divide by $L_n=2^{n/2-1}$, where $n\ge64$ is even. At $n=64$ the two ratios in
\eqref{eq:two-descent-estimates} are, respectively,
\[
 \frac{961\sqrt2}{2048}<1,
 \qquad \frac{58621}{65536}<1.
\]
Each ratio decreases as $n$ increases by two. Indeed the floor
in either exponent increases by at most one, while $L_n$
doubles. The polynomial growth ratio is bounded by
\[
 \frac{n+2}{n}\left(\frac n{n-2}\right)^2
 \frac{n-1}{n-3}
 \le\frac{33}{32}\left(\frac{32}{31}\right)^2\frac{63}{61}
 <\frac65<\sqrt2.
\]
Therefore the exponential decrease dominates the polynomial
growth. Finally, $n(n-1)<L_n$ at $n=64$ and the ratio decreases
for larger even $n$. This proves the strict inequality.
\end{proof}

\subsection{Removing the odd factors}

A finite class-two group splits into its Sylow subgroups.
The prime-wise index bounds exclude nonabelian odd factors at
the effective threshold.

\begin{lemma}[Separation of odd Sylow factors]\label{prop:odd-gap}
If a finite class-two group $G$ has a nonabelian odd-order Sylow
subgroup and $\om(G)\le n$ with $n\ge64$, then $a(G)<L_n$.
\end{lemma}
\begin{proof}
Suppose first that $G$ has exactly one nonabelian Sylow factor $P$.
Let $\beta$ be its reduced commutator map, so
$a(G)=a(\beta)$ and $\om(\beta)\le n$.
If at least one nonzero scalarization is degenerate, choose such a
$g=\lambda\beta$. Its radical $R$ is nonzero; write
$[V:R]=p^{2s}$ and $t=\tau(\beta|_R)$, and put
\[
 d=(p-1)(ps+1),\qquad B=\frac{p^{2s}-1}{d}.
\]
For every odd prime $p\ge3$ and $s\ge1$,
\begin{equation}\label{eq:odd-algebra}
 d-1\le p^{2s},\quad d\ge2p,\quad
 1+B\le\frac{2p^{2s}}{d},\quad
 \frac{2p^2}{d}\le\frac94.
\end{equation}
Indeed, $d-1\le p^{2s}$ holds at $s=1$, where
$d=p^2-1$, and persists by induction since
$d(s+1)-d(s)=p(p-1)\le(p^2-1)p^{2s}$;
the other bounds follow from $d\ge p^2-1$.
The covering inequality \eqref{eq:average} and
$a(\beta|_R)\le np^t$ give
\begin{equation}\label{eq:odd-cover}
 a(\beta)\le np^t(1+B).
\end{equation}

\emph{Degenerate scalarization: three radical cases.}
We distinguish $t=0$, $t=1$ and $t\ge2$.
If $t=0$, then $R$ is isotropic, and in fact
$a(\beta)\le1+nB\le1+np^{2s}$. By
Lemma~\ref{lem:scalarisation}, $n\ge p^2s+1$.
Since $\log p/p^2\le\log3/9$ for every prime $p\ge3$,
\begin{equation}\label{eq:odd-t0}
 a(\beta)\le1+n3^{2n/9}.
\end{equation}
If $t=1$, the coset cliques in the proof of
Proposition~\ref{prop:isotropic-index} give
$n\ge p(ps+1)\ge p^2s+1$. Now \eqref{eq:odd-algebra} implies
$p(1+B)\le p^{2s}$, so \eqref{eq:odd-cover} gives
\begin{equation}\label{eq:odd-t1}
 a(\beta)\le np^{2s}\le n3^{2n/9}.
\end{equation}
If $t\ge2$, the same proposition provides
$N_R(v)\ge2t$ for every coset in a scalar clique of
$ps+1$ cosets (also when $p=3,t=2$). Thus
$n\ge2t(ps+1)$, so
\[
 \frac n{2p}\ge ts+\frac tp,\qquad
 t+2s-2\le ts,
\]
where the latter inequality is equivalent to
$(t-2)(s-1)\ge0$. Using \eqref{eq:odd-cover} and
\eqref{eq:odd-algebra}, we obtain
\begin{align}
 a(\beta)
 &\le \frac{2np^2}{d}\,p^{t+2s-2}
 \le\frac94 n p^{n/(2p)}
 \le\frac94 n3^{n/6}.\label{eq:odd-tlarge}
\end{align}
The last step again uses that $\log p/p$ decreases for $p\ge3$.

\emph{Nondegenerate scalarizations.}
In the complementary case, every nonzero scalarization is
nondegenerate, so Lemma~\ref{lem:scalarisation} applies. For a nonscalar map, it gives
$s\ge2$, $n\ge p^2(s-1)+1$ and
\[
 a(\beta)\le\frac{p^{2s}-1}{p-1}
 \le\frac{p^2}{p-1}p^{2(s-1)}
 \le n3^{2n/9}.
\]
Here $p^2/(p-1)\le n$ follows already from
$n\ge p^2(s-1)+1\ge p^2+1$ because $s\ge2$.
If the map is scalar, the exact covering number is $p^s+1$.
For $p=3$, Lemma~\ref{lem:scalar} with $s=4q+r$, $0\le r\le3$,
gives
$n\ge15q+3r+1\ge(15s-5)/4$.
We claim that for $n\ge25$,
\begin{equation}\label{eq:odd-scalar}
 p^s\le2^{7(n-1)/15}.
\end{equation}
For $p=3$ and $s\le6$, the inequalities $3^5<2^8$ and
$n\ge25$ give $s\log_2 3<48/5<7(n-1)/15$.
For $p=3$ and $s\ge7$, they give
\[
 s\log_2 3<\frac85s
 \le\frac74s-\frac{21}{20}
 \le\frac7{15}(n-1).
\]
For $p\ge5$, the monotonicity of $\log p/p$ and
$5^3<2^7$ imply
$s\log_2p\le(7/15)ps\le(7/15)(n-1)$,
since $n\ge ps+1$. This proves \eqref{eq:odd-scalar}.

\emph{Comparison with the extremal scale.}
For $n\ge64$, normalize the bounds above by
$L_n\ge2^{n/2-1}$. The cases $t=0$, $t=1$, $t\ge2$,
the nonscalar nondegenerate case, and the scalar case yield,
respectively, the upper bounds
\[
 2^{1-n/2}+2n2^{-13n/90},\quad
 2n2^{-13n/90},\quad
 (9/2)n2^{-7n/30},\quad
 2n2^{-13n/90},\quad
 2^{1-n/2}\bigl(2^{7(n-1)/15}+1\bigr).
\]
Every one is strictly below one by
Lemma~\ref{lem:separation-numerics}.
If $G$ has at least two nonabelian Sylow factors, then
$a(G)\le n2^{7n/24}$; its normalized ratio is bounded by
$2n2^{-5n/24}<1$ by the same lemma.
This establishes the effective assertion.
\end{proof}

\begin{proof}[Proof of Proposition~\ref{prop:binary-reduction}]
Lemma~\ref{prop:nonclass2-gap} treats groups outside class two.
For class-two groups, Lemma~\ref{prop:odd-gap} excludes all
nonabelian odd Sylow factors. The remaining group has the
stated decomposition $G=P\times A$.
\end{proof}

Abelian direct factors preserve $a$ and $\om$, so only the
reduced commutator map of the class-two $2$-group remains.

\section{Separation of nonscalar alternating maps}\label{sec:parity}

After the reduction, it remains to distinguish the scalar
symplectic model from the other alternating maps on finite abelian
$2$-groups.

\begin{proposition}[Binary equality case]\label{parity:prop:binary}
For every integer $n\ge64$, every alternating map between
finite abelian $2$-groups with clique number at most $n$ satisfies
$a(\beta)\le L_n+1$.
Every reduced nonzero nonscalar map satisfies $a(\beta)<L_n$.
Equality holds precisely when the reduced map has target $\F_2$ and
domain $\F_2^{2\lfloor(n-1)/2\rfloor}$.
\end{proposition}

We prove the binary estimate for $n\ge64$; the odd bounds
in Theorem~\ref{thm:parity} follow by monotonicity.

The key is the radical $R$ of a scalarization. We first rule out
nonisotropic radicals and targets of exponent at least four.
When the target is elementary abelian, the interaction between
$R$ and $V/R$ either increases the clique number or has rank
at most one. The last possibility, a common image line, is resolved by
symplectic peeling followed by an elementary fibre cover.

Throughout the section, let $\om(\beta)\le n$ and reduce
$\beta$ by its radical, replacing its target by the subgroup
generated by its values. A nonzero reduced map is \emph{scalar}
if its target is $\F_2$; otherwise it is \emph{nonscalar}.
We retain the notation of Section~\ref{sec:class2}. Then
\begin{equation}\label{parity:eq:exponents}
 \exp V=\exp W:
\end{equation}
biadditivity gives one divisibility and $(\exp W)V\le\rad\beta=0$
gives the other. The zero map has $a(\beta)=1$.

\subsection{Scalarization}

For a nonzero scalarization put $R=\rad(\lambda\beta)$,
$[V:R]=2^{2s}$ and $t=\tau(\beta|_R)$.
By \eqref{eq:budget}, $a(\beta)\le n2^{2s+t}$.
The proof of Proposition~\ref{prop:isotropic-index} gives
$n\ge3(2s+1)$ for $t=1$ and $n\ge(t+1)(2s+1)$ for $t\ge2$.
If $R\ne0$ is isotropic, Lemma~\ref{lem:scalarisation} gives
\begin{equation}\label{parity:eq:iso-input}
 n\ge4s+1,\qquad
 a(\beta)\le1+\frac{n(4^s-1)}{2s+1}.
\end{equation}

\begin{lemma}[A nonisotropic scalar radical]\label{parity:lem:noniso}
If some nonzero scalarization has a nonisotropic radical, then
\[
 a(\beta)\le2n2^{n/3}.
\]
\end{lemma}
\begin{proof}
A nonisotropic radical has $t\ge1$.
For $t=1$, the bound $n\ge6s+3$ gives $2s+t\le n/3$.
For $t\ge2$,
\[
 (t+1)(2s+1)+3-3(2s+t)=2(s-1)(t-2)\ge0.
\]
Thus $2s+t\le n/3+1$, and the covering bound applies.
\end{proof}

The next lemma provides a different mechanism. If the domain
is not too large, a graph-colouring bound makes its isotropic
cover small without requiring a detailed analysis of the map.

\begin{lemma}[Colouring by independent sets]
\label{parity:lem:ramsey}
A graph with $N$ vertices and clique number at most $m$ has,
for every integer $k\ge2$,
\[
 \chi\le\frac Nk+\binom{m+k-1}{k-1}.
\]
The same bound holds for the isotropic covering number of an
alternating map with $N$ vectors and clique number at most $m$.
\end{lemma}
\begin{proof}
The Erd\H{o}s--Szekeres Ramsey recurrence
\cite{ErdosSzekeres1935} gives
$R(m+1,k)\le\binom{m+k-1}{k-1}$.
Repeatedly remove independent $k$-sets, each using one colour;
colour the fewer than $R(m+1,k)$ remaining vertices individually.
In a nonorthogonality graph, independent sets generate isotropic
subgroups, proving the last assertion.
\end{proof}

Write $Q_k(n)=\binom{n+k-1}{k-1}$.
The remaining obstruction to binary linear algebra is a target
of exponent at least four. Such an exponent forces additional
nonorthogonal vectors.

\begin{lemma}[Values of exponent at least four]\label{parity:lem:exponent}
If $W$ is not elementary abelian, some nonzero scalarization has
nonzero radical $R$ and rank $2s$ with $n\ge5s+1$.
\end{lemma}
\begin{proof}
Let $2^e=\exp W$, where $e\ge2$, and choose a projection
$\eta:W\to\mathbb Z/2^e\mathbb Z$ onto a cyclic direct factor of
maximum order. Put $b=\eta\beta$ and
$\lambda=\eta\pmod2$. Its composition with $\beta$ is nonzero.
By \eqref{parity:eq:exponents}, $V$ has exponent $2^e$.
If $b(x,y)$ is odd, then $\langle x,y\rangle\cong
(\mathbb Z/2^e\mathbb Z)^2$ and is nondegenerate for $b$.
Subtracting suitable multiples of $x,y$ from any vector splits off
this plane as a $b$-orthogonal direct summand. Repeating this process
splits off precisely $s$ planes, where $\rank(\lambda\beta)=2s$;
the remaining scalar form takes only even values.

The six-point construction modulo four is the $d=4$ case of
\cite[Theorem~4.1]{SarkarYoder2024}. Each plane contains such a clique,
with coefficient pairs
\[
 (1,0),(1,1),(1,2),(1,3),(0,1),(2,1).
\]
Modulo four, determinants within the first four points are nonzero
differences; those with $(0,1)$ are units, those with $(2,1)$ are
of the form $1-2t$, and the last pair has determinant $-2$.
Thus all determinants remain nonzero modulo $2^e$.
The orthogonal-sum construction in Lemma~\ref{lem:scalar} therefore gives
\[
 n\ge5s+1.
\]
The scalar radical $R$ is nonzero, since $2V\le R$ and $2V\ne0$.

\end{proof}

\subsection{From the radical to a common image}\label{sec:parity-radical}

We turn to elementary abelian targets. Equation
\eqref{parity:eq:exponents} makes $V$ elementary abelian as well,
so the argument now takes place over $\F_2$.
Fix a nonzero homomorphism $\lambda:W\to\F_2$ for which
$R=\rad(\lambda\beta)$ is nonzero and isotropic.
Let $U=V/R$, with its induced symplectic form $g$, and
write $\dim U=2s$. All spaces and maps below are over $\F_2$,
and $R^*,U^*$ denote the dual spaces.
For $u=v+R$, put $N(u)=N_R(v)$ and define
\[
 T_u:R\longrightarrow W,\qquad T_u(r)=\beta(v,r).
\]
This is well-defined, and $u\mapsto T_u$ is linear. Moreover,
\begin{equation}\label{parity:eq:coset-exact}
 N(u)=2^{\rank T_u}.
\end{equation}
Indeed, two vectors $v+r,v+r'$ are orthogonal exactly when
$T_u(r+r')=0$, so the nonorthogonality graph of the coset is complete
multipartite, with one part for each coset of $\ker T_u$.
The averaging argument of Lemma~\ref{lem:scalarisation} gives
\begin{equation}\label{parity:eq:average-input}
 n\ge\frac{2s+1}{4^s-1}\sum_{u\ne0}N(u).
\end{equation}

\begin{lemma}[An interaction of rank at least two]\label{parity:lem:ranktwo}
If some $T_u$ has rank at least two, then
\[
 n\ge\frac94(2s+1).
\]
\end{lemma}
\begin{proof}
Put $\mathcal M=\{T_u:u\in U\}$ and $d=\dim\mathcal M$.
Uniformly random $u\in U$ induces the uniform distribution on
$\mathcal M$. If $d=1$, the mean of $2^{\rank T_u}$ is at least
$(1+4)/2=5/2$.
If $d\ge2$, the zero matrix has probability at most $1/4$.
A $2\times2$ minor nonzero at some matrix is a nonzero Boolean
polynomial of degree at most two in coordinates on $\mathcal M$.
Such a polynomial is nonzero on at least one quarter of its inputs,
by the classical Reed--Muller minimum-distance bound
\cite[Section~II-B]{AbbeShpilkaYe2021}.
For completeness, the general degree-$D$ bound $2^{-D}$ follows
inductively by choosing a variable appearing in a highest-degree
monomial: on every input where its coefficient polynomial is one,
exactly one of the two extensions has value one.
Thus $\Pr(\rank T_u\ge2)\ge1/4$, and
\[
 \mathbb E\,2^{\rank T_u}
 \ge2-\Pr(T_u=0)+2\Pr(\rank T_u\ge2)\ge\frac94.
\]
Removing $u=0$, where the value is one, only increases this mean.
Apply \eqref{parity:eq:average-input}.
\end{proof}

\begin{lemma}[Linear spaces of maps of rank at most one]\label{parity:lem:rankone}
Suppose $\mathcal M\le\Hom(R,W)$ is nonzero and every member has
rank at most one. Then either all images lie in one fixed line of
$W$, or all maps vanish on one fixed hyperplane of $R$.
For the maps $T_u$ above, the latter alternative forces $\dim R=1$.
\end{lemma}
\begin{proof}
The first assertion is the classical classification of rank-one
matrix spaces \cite[Theorem~9.22, p.~113]{Harris1992}; we recall its short proof.
Write a nonzero rank-one map as $a\otimes f$, where
$a\in W$, $f\in R^*$, and $(a\otimes f)(x)=f(x)a$.
For any other nonzero map $b\otimes q$, its sum with $a\otimes f$
has rank two if both pairs $(a,b)$ and $(f,q)$ are independent.
Thus every map shares the image line or the kernel hyperplane of
$a\otimes f$. If maps of both distinct types occur, their sum
has rank two. Hence one of the two alternatives holds throughout.
In our setting $\bigcap_u\ker T_u=0$: an element in this intersection
is orthogonal to $R$ and to every representative of $V/R$, hence
belongs to $\rad\beta=0$. A common hyperplane is therefore zero.
\end{proof}

If $\dim R=1$, \eqref{parity:eq:iso-input} gives
$|V|\le2^{(n+1)/2}$, to which the colouring bound applies.
The remaining common-image alternative is
\begin{equation}\label{parity:eq:line}
 \beta(R,V)\le L=\langle\ell\rangle\le W,\qquad |L|=2.
\end{equation}
\begin{lemma}[Reduction to a common image]\label{parity:lem:other-strict}
Outside the common-image case, every reduced nonzero nonscalar binary
map has $a(\beta)<L_n$ for $n\ge64$.
\end{lemma}
\begin{proof}
We first list the cases to show that none is omitted. If the
target has exponent at least four, apply
Lemma~\ref{parity:lem:exponent}. Otherwise $V$ and $W$ are
binary vector spaces. Either every nonzero scalarization is
nondegenerate, or some scalarization has a nonzero radical $R$.
Such a radical is nonisotropic when $\beta|_R\ne0$.
If it is isotropic, either some $T_u$ has rank at least two or
all $T_u$ have rank at most one. In the latter situation,
Lemma~\ref{parity:lem:rankone} gives a common image line or a
common kernel hyperplane. Reducedness forces the latter to be
zero, so $\dim R=1$. These are all the possibilities.

\emph{Cases with extra clique growth.}
A nonisotropic scalar radical gives $a(\beta)\le2n2^{n/3}$ by
Lemma~\ref{parity:lem:noniso}.
If the target has exponent at least four, Lemma~\ref{parity:lem:exponent}
gives $s\le(n-1)/5$ and nonzero radical. If it is isotropic,
use the increase of $(4^s-1)/(2s+1)$ in \eqref{parity:eq:iso-input}.
With $j=\lfloor(n-1)/5\rfloor$, $n/(2j+1)<3$ for $n\ge64$, so
$a(\beta)\le1+3\,2^{2n/5}$. Otherwise use the preceding case.
For an interaction of rank at least two, put
$j=\lfloor2n/9-1/2\rfloor\ge s$ by Lemma~\ref{parity:lem:ranktwo}.
Again $n/(2j+1)<3$, giving $a(\beta)\le1+(3/2)2^{4n/9}$.
The three ratios to $2^{n/2-1}$ are bounded by
\[
 4n2^{-n/6},\qquad 2^{1-n/2}+6\,2^{-n/10},\qquad
 2^{1-n/2}+3\,2^{-n/18}.
\]
Lemma~\ref{lem:separation-numerics} puts each below one for $n\ge64$.

\emph{Cases with a small domain.}
If all scalarizations are nondegenerate and the map is nonscalar,
Lemma~\ref{lem:scalarisation} gives
$|V|\le4^{\lfloor(n+3)/4\rfloor}\le4L_n$.
The common-kernel case has $|V|\le2^{(n+1)/2}\le4L_n$.
Lemma~\ref{parity:lem:ramsey} with $k=6$ gives
\[
 \frac{a(\beta)}{L_n}\le\frac23+\frac{\binom{n+5}{5}}{L_n}
 \le\frac23+\frac{2^{1-n/2}(n+5)^5}{120}.
\]
Lemma~\ref{lem:separation-numerics} bounds the error by $1/6$
for $n\ge64$, proving the strict inequality in every case.
\end{proof}

\subsection{Covering the common-image fibres}\label{sec:parity-pairing}

Only the common-image case remains. Put $r=\dim R$ and
define the linear map $f:U\to R^*$ by
\[
 \beta(v,x)=f(v+R)(x)\ell,\qquad x\in R.
\]
Its transpose $R\to U^*$ is injective, because $\rad\beta=0$.
Hence $f$ is surjective, $r\le2s$, and
\[
 C=\ker f,\qquad c=\dim C=2s-r.
\]
Since $R=\rad(\lambda\beta)$, we have $\lambda(\ell)=0$.

The kernel dimension $c$ measures the number $2^c$ of cosets
of $R$ having a prescribed interaction with $R$. The fibres of
$f$ group these cosets together. Crucially, the zero fibre is
itself a subgroup, so we can colour all of its elements at once.
In a nonzero fibre, two representatives of each coset suffice.

\begin{lemma}[Covering interaction fibres]\label{parity:lem:fibre-colouring}
In the common-image case,
\begin{equation}\label{parity:eq:coset-count}
 a(\beta)\le2\cdot4^s-2^c.
\end{equation}
Writing $Q_3(n)=\binom{n+2}{2}$, one also has
\begin{equation}\label{parity:eq:fibre-effective}
 \frac{a(\beta)}{4^s}
 \le\frac23+\frac{2^{-r}}3+Q_3(n)2^{-c}.
\end{equation}
\end{lemma}
\begin{proof}
Let $\pi:V\to U$ be the quotient map. The zero fibre
$V_C=\pi^{-1}(C)$ has $2^c$ cosets of $R$, and each is contained
in an isotropic subgroup $\langle R,v_u\rangle$.

For each nonzero $\varphi\in R^*$ choose $x\in R$ with
$\varphi(x)=1$. If $v_u$ is a representative of
$u\in f^{-1}(\varphi)$, then
\[
 S_\varphi=\{v_u,v_u+x:u\in f^{-1}(\varphi)\}
\]
consists of $2^{c+1}$ vectors orthogonal to $\ker\varphi$.
An independent set in its nonorthogonality graph, together with
$\ker\varphi$, spans an isotropic subgroup covering the
corresponding cosets. By Lemma~\ref{parity:lem:ramsey} with $k=3$,
this fibre costs at most $2^{c+1}/3+Q_3(n)$ subgroups.
Adding the $2^r-1$ nonzero fibres and dividing by
$4^s=2^{c+r}$ gives \eqref{parity:eq:fibre-effective}.
Alternatively, covering each coset separately costs one subgroup
in the zero fibre and two in every other fibre, proving
\eqref{parity:eq:coset-count}.
\end{proof}

For small $c$, we extract a dual-number plane while preserving
the interaction-kernel dimension in its orthogonal complement.
The required decomposition is proved explicitly below; compare
\cite[Sections~3.4--3.6]{BrooksbankMaglioneWilson2017}.

\begin{lemma}[Symplectic peeling]\label{parity:lem:dual}
In the common-image case, with $s$ and $c$ as above,
\begin{equation}\label{parity:eq:dual}
 \om(\beta)\ge 4s+1+\max\{s-c,0\}.
\end{equation}
\end{lemma}
\begin{proof}
Project the target onto two scalar coordinates
$(\lambda,\mu)$, choosing $\mu(\ell)=1$; write
$g=\lambda\beta$ and $h=\mu\beta$.
Nonorthogonality for $(g,h)$ implies nonorthogonality for $\beta$.
We have $R=\rad g$, $h(R,R)=0$, and $g$ induces a nondegenerate
form on $U=V/R$ of dimension $2s$. The interaction identifies
\begin{equation}\label{eq:peeling-annihilator}
 R\overset{\sim}{\longrightarrow}\operatorname{Ann}(C)\subset U^*,
 \qquad x\longmapsto (u\mapsto h(x,\widetilde u)).
\end{equation}
Indeed the map is well-defined and annihilates $C=\ker f$;
it is injective because $\beta(R,V)\le\langle\ell\rangle$
and $\rad\beta=0$. Both sides have dimension
$r=2s-c$, so it is an isomorphism.

Suppose first that $c<s$. The restriction of $g$ to
$C^{\perp g}$ has radical $C\cap C^{\perp g}$, of dimension
at most $c$. Consequently its nondegenerate part has dimension
at least $2s-2c\ge2$. Choose a nondegenerate symplectic plane
$A\le C^{\perp g}$ and let $B=A^{\perp g}$.
Then $U=A\perp_g B$ and $C\le B$. Under
\eqref{eq:peeling-annihilator}, set
\[
 R_A=\operatorname{Ann}(B),\qquad
 R_B=\operatorname{Ann}(A)\cap\operatorname{Ann}(C),
\]
where the annihilators are understood as subspaces of $R$.
They give $R=R_A\oplus R_B$, with $\dim R_A=2$ and
$\dim R_B=r-2$, and the $h$-pairing between $A$ and $R_A$
is perfect. The cross pairings $h(B,R_A)$ and $h(A,R_B)$ vanish.

We now choose lifts of $A$ and $B$ so that they are orthogonal
for \emph{both} $g$ and $h$. Start with arbitrary linear sections
$\sigma_A:A\to V$, $\sigma_B:B\to V$.
Since the $h$-pairing $A\times R_A$ is perfect, we may add a
linear map $A\to R_A$ to $\sigma_A$ so that
$h(\sigma_A(A),\sigma_A(A))=0$: explicitly, split the alternating
form on $A$ as $q(a,b)+q(b,a)$ and use perfection to realize
$q$ by the correction. We may then add a linear map
$B\to R_A$ to $\sigma_B$ so that
$h(\sigma_A(A),\sigma_B(B))=0$.
Neither correction alters $g$ or the stated pairings with $R$.
It follows that
\[
 V=T\oplus H,\qquad
 T=\sigma_A(A)\oplus R_A,\quad
 H=\sigma_B(B)\oplus R_B,
\]
is an orthogonal direct sum for the two-coordinate map $(g,h)$.

On $T$, its two alternating forms are, after a symplectic
identification of $R_A$ with $A$,
\[
 ((a,x),(b,y))\longmapsto
 \bigl(g(a,b),\,g(a,y)+g(x,b)\bigr).
\]
This is one symplectic plane over the dual numbers
$\F_2[\eps]/(\eps^2)$. It has a six-point nonorthogonal set:
\[
 \{(1,t):t\in\F_2[\eps]/(\eps^2)\}
 \ \cup\ \{(0,1),(\eps,1)\}.
\]
On $H$, the scalar form $g$ has rank $2(s-1)$ and radical
$R_B$. The restriction of $h$ pairs $R_B$ injectively with $B$:
an element of $R_B$ annihilating $B$ would also annihilate $A$
and hence, by \eqref{eq:peeling-annihilator}, be zero.
To check reducedness explicitly, any vector orthogonal to $H$
for both $g$ and $h$ first lies in $R_B$, because $g|_B$ is
nondegenerate, and then vanishes by this injective pairing.
Thus the residual two-coordinate map is reduced, its scalar
radical $R_B$ is isotropic, and its interaction is again in the
common image line. The interaction map on $B$ is onto $R_B^*$, so its
kernel has dimension $2(s-1)-(r-2)=c$.
Thus $H$ is another common-image instance with parameters
$(s-1,c)$ (or the zero space if $s=1,c=0$).

For alternating maps on orthogonal summands, cliques of sizes
$a$ and $b$ combine into one of size $a+b-1$ by the construction
$(X\times\{y_0\})\cup(\{0\}\times(Y\setminus\{y_0\}))$.
The six-point set therefore gives
\[
 \om(\beta)\ge 5+\om((g,h)|_H).
\]
Repeat the peeling exactly $s-c$ times. The terminal remainder
has scalar rank $2c$ and isotropic radical of dimension $c$.
If $c>0$, the isotropic-radical gain in
Lemma~\ref{lem:scalarisation} gives a clique of size $4c+1$;
if $c=0$, the remainder is zero and its one-point set plays
the same role. Hence
\[
 \om(\beta)\ge5(s-c)+4c+1=5s-c+1.
\]
For $c\ge s$, the original isotropic-radical estimate
$\om(\beta)\ge4s+1$ already proves the claimed inequality.
\end{proof}

The peeling bound controls the two cases in which the elementary
coset estimate is not yet sufficient.

\begin{lemma}[Common-image separation]\label{parity:lem:fibre-threshold}
In the common-image case, $a(\beta)<L_n$ for every $n\ge64$.
\end{lemma}
\begin{proof}
As the scalar radical is isotropic, $n\ge4s+1$; by
Lemma~\ref{parity:lem:dual} we have
\begin{equation}\label{eq:peeling-clique}
 n\ge4s+1+\max\{s-c,0\}.
\end{equation}
Write $n=4s+\delta$, where $\delta\ge1$.
If $\delta\ge3$, equation~\eqref{parity:eq:coset-count} yields
$a(\beta)<2\cdot4^s\le L_n$.
If $\delta\in\{1,2\}$, then $s\ge16$ because $n\ge64$,
and \eqref{eq:peeling-clique} forces $c\ge s-1$.
Applying \eqref{parity:eq:fibre-effective} gives
\[
 \frac{a(\beta)}{4^s}
 \le\frac23+\frac{2^{-r}}3+\binom{n+2}{2}2^{-c}
 \le\frac56+\binom{4s+4}{2}2^{1-s}<1.
\]
For the final inequality, the last term is $2278/32768<1/6$
at $s=16$ and decreases thereafter. In these two cases
$L_n=4^s$, proving the claim.
\end{proof}

\begin{proof}[Proof of Proposition~\ref{parity:prop:binary}]
Reduce the map. Lemmas~\ref{parity:lem:other-strict}
and~\ref{parity:lem:fibre-threshold} treat every nonzero nonscalar
case; the zero map has $a=1$.
For a scalar map, Lemma~\ref{lem:extraspecial} gives
$\om(\beta)=2s+1$ and $a(\beta)=2^s+1$.
Thus $s\le\lfloor(n-1)/2\rfloor$, with equality precisely
under the stated dimension condition.
\end{proof}

\section{Proof of the main theorem and open questions}\label{sec:completion}

\begin{proof}[Proof of Theorem~\ref{thm:parity}]
By Lemma~\ref{lem:finite-reduction}, it suffices to consider finite groups.
At even bounds $n\ge64$, Proposition~\ref{prop:binary-reduction}
either gives $a(G)<L_n$ or reduces to a class-two $2$-group
without changing $a$ or $\om$.
Proposition~\ref{parity:prop:binary} then gives $a(G)\le L_n+1$,
with equality precisely for a scalar map of dimension
$2m$, where $m=\lfloor(n-1)/2\rfloor$.
Lemma~\ref{lem:extraspecial} attains this value.
For odd $n\ge63$, apply the even case at $n+1$, noting that
$L_{n+1}=L_n$. This gives the remaining bounds and the same
characterization of equality.

For a class-two group the commutator map is reduced, so the equality
condition is exactly $G'\cong C_2$ and $G/Z(G)\cong\F_2^{2m}$.
Conversely, a normal subgroup of order two is central, and these
conditions give the required nondegenerate scalar form.
Finite reduction preserves both structural invariants.

\end{proof}

\subsection*{Open questions}

The threshold $63$ in Theorem~\ref{thm:parity} is sufficient,
but need not be optimal. Indeed the formula fails at $n=4,6$.
For $q\in\{3,5\}$, the dihedral group of order $2q$ has
clique number $q+1$: its $q$ reflections and any nonidentity
rotation form a clique, while a clique contains at most one
rotation and at most $q$ reflections. Since $a(G)\ge\om(G)$,
\[
 h(4)\ge4>3=L_4+1,\qquad
 h(6)\ge6>5=L_6+1.
\]

\paragraph{Question 1 (Optimal threshold).}
Determine the least integer
\[
 N_* = \min\bigl\{N\ge1:\ h(n)=2^{\lfloor(n-1)/2\rfloor}+1
                    \text{ for every }n\ge N\bigr\}.
\]
The proven bounds give $7\le N_*\le63$. In fact $N_*$ is odd:
if equality holds at $2m+2$, monotonicity and the extraspecial
lower bound imply
\[
 2^m+1\le h(2m+1)\le h(2m+2)=2^m+1.
\]
Thus $N_*\in\{7,9,11,\ldots,63\}$.

\paragraph{Question 2 (Small-parameter values).}
Determine $h(n)$ for $4\le n<63$. In particular, identify
all parameters where
$h(n)\ne 2^{\lfloor(n-1)/2\rfloor}+1$ and, where possible,
classify the groups attaining $h(n)$.

\appendix
\section{Numerical estimates}\label{app:numerics}

We verify the explicit inequalities used at the even threshold
$n=64$. All comparisons are elementary and are included to make
the effective constants transparent.

We will use the elementary logarithmic estimate
\begin{equation}\label{eq:log-lower}
 \log 2>2\left(\frac13+\frac1{81}\right)
 =\frac{56}{81}>\frac{11}{16},
\end{equation}
obtained by keeping the first two positive terms in the series
$\log 2=2\sum_{j\ge0}3^{-(2j+1)}/(2j+1)$.

\begin{lemma}[Effective separation bounds]\label{lem:separation-numerics}
For every integer $n\ge64$, each expression in the left column is
less than the constant in the right column.
\[
\begin{array}{c|c}
 \text{expression}&\text{upper bound}\\ \hline
 2^{1-n/2}+2n2^{-13n/90}&1/4\\
 (9/2)n2^{-7n/30}&1/32\\
 2^{1-n/2}(2^{7(n-1)/15}+1)&1/2\\
 2n2^{-5n/24}&1/64\\
 4n2^{-n/6}&1/4\\
 2^{1-n/2}+6\,2^{-n/10}&1/8\\
 2^{1-n/2}+3\,2^{-n/18}&1/2\\
 2^{1-n/2}(n+5)^5/120&1/6
\end{array}
\]
\end{lemma}
\begin{proof}
Every summand decreases on $[64,\infty)$: for the polynomial
factors use the logarithmic derivative
$d/n-\alpha\log2$ of $n^d2^{-\alpha n}$, or
$5/(n+5)-(\log2)/2$ for the last row.
The bound $\log2>11/16$ from \eqref{eq:log-lower}
suffices in each case.

It remains to substitute $n=64$.
For the first row, $2^{416/45}>2^9(9/8)$, since
$11/45>1/5$ and $(9/8)^5<2$; hence the value is less than
$2^{-31}+2/9<1/4$.
For the second, $448/30>14$ gives
$288/2^{14}<1/32$.
For the third, $2^{-8/5}<2^{-3/2}<3/8$, leaving
$3/8+2^{-31}<1/2$.
The fourth equals $2^{-19/3}<1/64$.
The next three follow respectively from
$64/6>10$, $64/10>6$ and $64/18>3$.
Finally, $69^5/(120\,2^{31})<80^5/(120\,2^{31})<1/6$.
\end{proof}

\section*{Acknowledgements and use of AI}
ChatGPT (OpenAI) and Aristotle (Harmonic) were used for exploratory
proof checking and consistency review; ChatGPT also assisted with
language and \LaTeX{} editing. All mathematical arguments used in this
article are stated and proved within the text. Responsibility for the
claims and their verification rests with the author.

\end{document}